\documentclass{article}%
\usepackage{amsmath}
\usepackage{amsfonts}
\usepackage{amssymb}
\usepackage{graphicx}%
\providecommand{\U}[1]{\protect\rule{.1in}{.1in}}
\newtheorem{theorem}{Theorem}

\newtheorem{conclusion}[theorem]{Conclusion}

\newtheorem{example}{Example}

\newtheorem{remark}{Remark}

\newenvironment{proof}[1][Proof]{\noindent\textbf{#1.} }{\ \rule{0.5em}{0.5em}}
\begin{document}

\title{A few results concerning the Schur stability of the Hadamard powers and the
Hadamard products of complex polynomials}
\author{Micha\l \ G\'{o}ra\thanks{e-mail: gora@agh.edu.pl; This research work was partially supported by the Faculty of Applied
Mathematics AGH UST statutory tasks (grant no. 11.11.420.004) within subsidy
of Ministry of Science and Higher Education.} \\AGH University of Science and Technology,\\ Faculty of Applied Mathematics,\\ al.
Mickiewicza 30, 30-059 Krak\'{o}w, Poland}
\date{}
\maketitle

\begin{abstract}
For a complex polynomial
\[
f\left(  s\right)  =s^{n}+a_{n-1}s^{n-1}+\ldots+a_{1}s+a_{0}%
\]
and for a rational number $p$, we consider the Schur stability problem of the
$p$-th Hadamard power of $f$
\[
f^{\left[  p\right]  }\left(  s\right)  =s^{n}+a_{n-1}^{p}s^{n-1}+\ldots
+a_{1}^{p}s+a_{0}^{p}\text{.}%
\]
We show that there exist two numbers $p^{\ast}\geq0\geq p_{\ast}$ such that
$f^{\left[  p\right]  }$ is Schur stable for every $p>p^{\ast}$ and is not
Schur stable for $p<p_{\ast}$ (or vice versa, depending on $f$). Also, we give
simple sufficient conditions for the Schur stability of the Hadamard product
of two complex polynomials. Numerical examples complete and illustrate the results.

\end{abstract}

\section{Introduction}

Over two decades ago, in 1996, Garloff and Wagner \cite{Gar-Wag} provided an
interesting property of the Hurwitz stable polynomials. They proved that the
Hadamard product (i.e. element-wise multiplication) of two real Hurwitz stable
polynomials is again Hurwitz stable. An immediate consequence of the
Garloff--Wagner result is that the stability of $f$ implies that of
$f^{\left[  p\right]  }$, the $p$-th Hadamard power of $f$, for every positive
integer $p$. Gregor and Ti\v{s}er \cite{Gre-Tis} claimed that even more is
true, that is, that the $p$-th Hadamard power of a Hurwitz stable polynomial
is Hurwitz stable for every real power $p>1$. Unfortunately, as Bia\l as and
Bia\l as-Cie\.{z} proved in their recent work \cite{Bia-Cie}, they were wrong,
i.e. for a stable polynomial $f$, the polynomial $f^{\left[  p\right]  }$ does
not need to be Hurwitz stable for $p>1$.

Motivated by the work of Bia\l as and Bia\l as-Cie\.{z}, we will focus the
attention on the Schur stability problem of the Hadamard powers of complex
polynomials. It is known that the result of Garloff and Wagner does not extend
neither to the complex case nor to the class of the Schur stable polynomials
(see Bose and Gregor \cite{Bos} or again Garloff and Wagner \cite{Gar-Wag}).
The main aim of this work is to show that for a very wide class of complex
polynomials including, among others, unstable elements, it is possible to find
two numbers $p^{\ast}\geq0\geq$ $p_{\ast}$ depending on $f$ and such that the
$p$--th Hadamard power of $f$ is Schur stable for every $p>p^{\ast}$ and is
not Schur stable for every $p<p_{\ast}$ (or vice versa). Some attention is
also paid to possibility of construction of families of the Schur stable
polynomials with complex coefficients that are closed under the Hadamard
multiplication. The obtained results complete and generalize those given by
Garloff and Wagner~\cite{Gar-Wag}, Gregor and Ti\v{s}er~\cite{Gre-Tis} and
Bia\l as and Bia\l as-Cie\.{z}~\cite{Bia-Cie}.

\section{Preliminary results}

\subsection{Basic notations}

We use standard notation: $\mathbb{Q}$, $\mathbb{R}$ and $\mathbb{C}$ stand
for the set of rational numbers, real numbers and complex numbers,
respectively; $\mathbb{\pi}_{n}\left(  \mathbb{C}\right)  $ stands for the
family of $n$--th degree monic polynomials with complex coefficients;
$\left\vert \cdot\right\vert $ denotes the moduli of a complex number and $i$
stands for the imaginary unit.

\subsection{Stable polynomials}

A polynomial is said to be Schur stable (shortly \textit{stable}) if all its
zeros lie in the open unit disc. From among many sufficient conditions for the
stability of a polynomial we recall the one following from the bound for the
moduli of zeros of polynomials given by Fujiwara \cite{Fujiwara}: a polynomial
$f\in\mathbb{\pi}_{n}\left(  \mathbb{C}\right)  $ of the form $f\left(
s\right)  =s^{n}+a_{n-1}s^{n-1}+\ldots+a_{1}s+a_{0}$ is stable if it satisfies
\textit{the stability condition, }i.e.\textit{ there exist }$\left\{
\lambda_{k}\right\}  $\textit{, a sequence of positive numbers whose sum does
not exceed~}$1$, \textit{such that the following condition holds} \textit{ }%
\begin{equation}
\left\vert a_{k}\right\vert <\lambda_{k},\text{\quad\textit{for} }k\in N_{f}
\label{eq2}%
\end{equation}
\textit{where} $N_{f}=\{k\in\left\{  0,\ldots,n-1\right\}  :a_{k}\neq
0\}$\label{eq211} (the proof of sufficiency of (\ref{eq2}) for the stability
of $f$ can be easily derived from Fujiwara's work \cite{Fujiwara}, but for the
sake of completeness of the article we present it in Appendix A). It seems to
be interesting and important to note that the stability condition (\ref{eq2})
is \textit{sharp} in the sense that for every $\varepsilon\geq0$ and for every
sequence of positive numbers $\lambda_{0},\ldots,\lambda_{n-1}$ summing up to
$1+\varepsilon$, the polynomial $s\rightarrow s^{n}-%
%TCIMACRO{\tsum \nolimits_{k=0}^{n-1}}%
%BeginExpansion
{\textstyle\sum\nolimits_{k=0}^{n-1}}
%EndExpansion
\lambda_{k}s^{k}$ is unstable.

\subsection{The Hadamard product and the Hadamard powers of polynomials}

For two polynomials $f,g\in\mathbb{\pi}_{n}\left(  \mathbb{C}\right)  ,$%
\begin{equation}%
\begin{array}
[c]{l}%
f\left(  s\right)  =s^{n}+a_{n-1}s^{n-1}+\ldots+a_{1}s+a_{0}\\
g\left(  s\right)  =s^{n}+b_{n-1}s^{n-1}+\ldots+b_{1}s+b_{0}%
\end{array}
\label{w01}%
\end{equation}
we define their Hadamard product $f\circ g$ as an $n$--th degree polynomial of
the form%
\[
\left(  f\circ g\right)  \left(  s\right)  =s^{n}+a_{n-1}b_{n-1}s^{n-1}%
+\ldots+a_{1}b_{1}s+a_{0}b_{0}\text{.}%
\]
In turn, for $p\in\mathbb{Q}$ the polynomial $f^{\left[  p\right]  }$%
\[
f^{\left[  p\right]  }\left(  s\right)  =s^{n}+a_{n-1}^{p}s^{n-1}+\ldots
+a_{1}^{p}s+a_{0}^{p},
\]
is called the $p$-th Hadamard power of $f$ (we put, by definition, that
$0^{p}=0$ for $p\in\mathbb{Q}$). If $p$ is an integer then $f^{[p]}$ is a
polynomial. However, if $p$ is a non-integer rational number, say $p=k/m$ with
$k$ and $m$ relatively prime integers, then $p$--th power of the complex
number $a_{i}$ is not a number but is a set of $m$ complex numbers whose
$m$-th power gives $a_{i}^{k}$. In other words, for $a_{j}=\left\vert
a_{j}\right\vert \left(  \cos\alpha_{j}+i\sin\alpha_{j}\right)  $ we have
$a_{j}^{p}=\left\{  a_{j,0}^{p},\ldots,a_{j,m-1}^{p}\right\}  ,$ where
\begin{equation}
a_{j,l}^{p}=\left\vert a_{j}\right\vert ^{p}\left(  \cos\left(  p\alpha
_{j}+2\pi l/m\right)  +i\sin\left(  p\alpha_{j}+2\pi l/m\right)  \right)
\text{,} \label{w03}%
\end{equation}
for $l=0,\ldots,m-1$. In that case, the $p$-th Hadamard power of a polynomial
should be understood as a set of $m^{n}$ polynomials
\begin{equation}
s\rightarrow s^{n}+a_{n-1,l_{n-1}}^{p}s^{n-1}+\ldots+a_{1,l_{1}}%
^{p}s+a_{0,l_{0}}^{p} \label{w03a}%
\end{equation}
whose coefficients are calculated as in (\ref{w03}).

\section{Main results}

\subsection{The Schur stability of the Hadamard powers of a polynomial}

Let, for $f$ as in (\ref{w01}) and for $N_{f}\ $ as on page \pageref{eq211},
\[
\Lambda_{f}=\left\{  \left\{  \lambda_{k}\right\}  _{k\in N_{f}}:\lambda
_{k}\in(0,1],\text{ }%
%TCIMACRO{\tsum \nolimits_{k\in N_{f}}}%
%BeginExpansion
{\textstyle\sum\nolimits_{k\in N_{f}}}
%EndExpansion
\lambda_{k}\leq1\right\}  \text{.}%
\]

We are now ready to formulate the main result of this work.

\begin{theorem}
\label{th1} For $f\in\mathbb{\pi}_{n}\left(  \mathbb{C}\right)  $ as in
(\ref{w01}) the following hold:

\begin{enumerate}
\item[(a)] if $N_{f}$ is non-empty and $\left\vert a_{k}\right\vert <1$ for
$k\in N_{f}$, then $f^{\left[  p\right]  }$ is Schur stable for every
$p>p_{\max}^{\ast}\geq0$ where%
\begin{equation}
p_{\max}^{\ast}=\inf\limits_{\left\{  \lambda_{k}\right\}  \in\Lambda_{f}}%
\max_{k\in N_{f}}\frac{\ln\lambda_{k}}{\ln\left\vert a_{k}\right\vert };
\label{wth1}%
\end{equation}

\item[(b)] if $N_{f}$ is non-empty and $\left\vert a_{k}\right\vert >1$ for
$k\in N_{f}$, then $f^{\left[  p\right]  }$ is Schur stable for every
$p<p_{\min}^{\ast}\leq0$ where%
\begin{equation}
p_{\min}^{\ast}=\sup_{\left\{  \lambda_{k}\right\}  \in\Lambda_{f}}\min_{k\in
N_{f}}\frac{\ln\lambda_{k}}{\ln\left\vert a_{k}\right\vert }; \label{wth2}%
\end{equation}

\item[(c)] if $N_{f}$ is empty, then $f^{\left[  p\right]  }$ is stable for
every $p\in\mathbb{Q}$.
\end{enumerate}
\end{theorem}

\begin{proof}
If $N_{f}$ is empty then the result is obvious. Suppose thus that $N_{f}$ is
non-empty. We can restrict our considerations to the real polynomial
$s\rightarrow s^{n}+%
%TCIMACRO{\tsum \nolimits_{k=0}^{n-1}}%
%BeginExpansion
{\textstyle\sum\nolimits_{k=0}^{n-1}}
%EndExpansion
\left\vert a_{k}\right\vert s^{k}$ and its $p$--th power $s\rightarrow s^{n}+%
%TCIMACRO{\tsum \nolimits_{k=0}^{n-1}}%
%BeginExpansion
{\textstyle\sum\nolimits_{k=0}^{n-1}}
%EndExpansion
\left\vert a_{k}\right\vert ^{p}s^{k}$ being a polynomial with nonnegative
coefficients. Indeed, if the real polynomial $s\rightarrow$ $s^{n}+%
%TCIMACRO{\tsum \nolimits_{k=0}^{n-1}}%
%BeginExpansion
{\textstyle\sum\nolimits_{k=0}^{n-1}}
%EndExpansion
r_{k}s^{k}$ with nonnegative coefficients $r_{0},\ldots,r_{n-1}$ satisfies the
stability condition, then every complex polynomial whose $k$--th coefficient
has the moduli equal to $r_{k}$ (for $k=0,\ldots,n-1$) satisfies it too. In
other words, the polynomial $s\rightarrow s^{n}+%
%TCIMACRO{\tsum \nolimits_{k=0}^{n-1}}%
%BeginExpansion
{\textstyle\sum\nolimits_{k=0}^{n-1}}
%EndExpansion
\left\vert a_{k}\right\vert ^{p}s^{k}$ satisfies the stability condition if
and only if each polynomial of the form (\ref{w03a}), and thus $f^{[p]}$, does.

Let $\left\{  \lambda_{k}\right\}  $ be an arbitrary element of $\Lambda_{f}$.
The stability condition applied to the polynomial $f^{\left[  p\right]  }$
gives%
\begin{equation}
p\ln\left\vert a_{k}\right\vert <\ln\lambda_{k}\text{{}} \label{w04}%
\end{equation}
for $k\in N_{f}$. If $\left\vert a_{k}\right\vert <1$ for $k=0,\ldots,n-1$,
then (\ref{w04}) leads to%
\begin{equation}
p>\max_{k\in N_{f}}\frac{\ln\lambda_{k}}{\ln\left\vert a_{k}\right\vert }
\label{w05a}%
\end{equation}
and in case $\left\vert a_{k}\right\vert >1$ for $k\in N_{f}$, it leads to%
\begin{equation}
p<\min_{k\in N_{f}}\frac{\ln\lambda_{k}}{\ln\left\vert a_{k}\right\vert
}\text{.} \label{w05b}%
\end{equation}
Since it is sufficient for the stability of $f^{\left[  p\right]  }$ that
inequality (\ref{w05a}) or inequality (\ref{w05b}) holds for at least one
sequence $\left\{  \lambda_{k}\right\}  $, we can repeat the same for every
$\left\{  \lambda_{k}\right\}  \in\Lambda_{f}$ and take in (\ref{w05a})
infimum over all $\left\{  \lambda_{k}\right\}  $ and supremum over all
$\left\{  \lambda_{k}\right\}  $ in (\ref{w05b}). This yields to (a) and (b).
\end{proof}

\begin{remark}
\label{rem1}As it is known (see Example 5.3 in Saydy et al. \cite{Saydy}), in
the entire family of real polynomials having all roots in the closed unit
disc, the so-called guardian map%
\[
\Phi:f\rightarrow f\left(  1\right)  f\left(  -1\right)  \det D_{f},
\]
where $D_{f}$ is some real matrix of order $n-1$ formed from the coefficients
of $f$, vanishes if and only if $f$ is unstable (has a root on the unit
circle). Thus, when for a real polynomial $f$ with nonnegative coefficients
there exists, as in Theorem~\ref{th1}, a number $p^{\ast}$ for which
$f^{\left[  p\right]  }$ is stable for $p>p^{\ast}$ (or for $p<p^{\ast}$) then
the minimal (maximal) value of such $p^{\ast}$ can be calculated as the
maximal (minimal) real zero of the function%
\[
\Phi_{f}:p\rightarrow f^{\left[  p\right]  }\left(  1\right)  f^{\left[
p\right]  }\left(  -1\right)  \det D_{f^{\left[  p\right]  }}\text{.}%
\]
In case of a complex polynomial $f$ and its integer Hadamard powers, such
$p^{\ast}$, if any, can be calculated as the maximal (minimal) real zero of
the function%
\[
\tilde{\Phi}_{f}:p\rightarrow f_{\operatorname{Re}}^{\left[  p\right]
}\left(  1\right)  f_{\operatorname{Re}}^{\left[  p\right]  }\left(
-1\right)  \det D_{f_{\operatorname{Re}}^{\left[  p\right]  }}\text{,}%
\]
where $f_{\operatorname{Re}}=\bar{f}\cdot f$ and $\bar{f}$ is a polynomial
whose coefficients are complex conjugates of these of $f$.
\end{remark}

The next theorem shows that the assumptions of Theorem~\ref{th1} are relevant.

\begin{theorem}
\label{th2}For $f\in\mathbb{\pi}_{n}\left(  \mathbb{C}\right)  $ as in
(\ref{w01}) the following hold:

\begin{enumerate}
\item[(a)] if the set $N_{f}$ is non-empty and $k^{\ast}=\min\left\{  k:k\in
N_{f}\right\}  $, then $f^{\left[  p\right]  }$ is not Schur stable for every
$p\leq0$ if $\left\vert a_{k^{\ast}}\right\vert \leq1$ and for every $p\geq0$
if $\left\vert a_{k^{\ast}}\right\vert \geq1$;

\item[(b)] if the set $N_{f,1+}=\left\{  k\in N_{f}:\left\vert a_{k}%
\right\vert >1\right\}  $ is non-empty, then $f^{\left[  p\right]  }$ is not
Schur stable for every $p\geq\beta_{\max}^{\ast}\geq0$, where%
\[
\beta_{\max}^{\ast}=\min_{k\in N_{f,1+}}\frac{\ln\binom{n}{k}}{\ln\left\vert
a_{k}\right\vert }\text{;}%
\]

\item[(c)] if the set $N_{f,1-}=\left\{  k\in N_{f}:\left\vert a_{k}%
\right\vert <1\right\}  $ is non-empty, then $f^{\left[  p\right]  }$ is not
Schur stable for every $p\leq\beta_{\min}^{\ast}\leq0$ where%
\[
\beta_{\min}^{\ast}=\max_{k\in N_{f,1-}}\frac{\ln\binom{n}{k}}{\ln\left\vert
a_{k}\right\vert }\text{.}%
\]

\end{enumerate}
\end{theorem}

\begin{proof}
Since $a_{k^{\ast}}$ is, with accuracy to the sign, a product of all nonzero
roots of the polynomial $f$, condition (a) is obvious. To prove (b) and (c),
recall that a necessary condition for the stability of $f$ is that $\left\vert
a_{k}\right\vert <\binom{n}{k}$, for $k=0,\ldots,n-1$. In other words, if for
some $k\in N_{f}$,
\begin{equation}
p\ln\left\vert a_{k}\right\vert \geq\ln\binom{n}{k}, \label{w06}%
\end{equation}
then $f^{\left[  p\right]  }$ is not stable. For $k\in N_{f,1+}$ (\ref{w06})
follows from $p\geq\beta_{\max}^{\ast}$ proving (b) and for $k\in N_{f,1-}$
from $p\leq\beta_{\min}^{\ast}$ proving (c).
\end{proof}

\subsection{The Schur stability of the Hadamard product of polynomials}

Now, we will focus the attention on the stability of the Hadamard product
$f\circ g$ of two complex polynomials $f,g\in\mathbb{\pi}_{n}\left(
\mathbb{C}\right)  $. As mentioned in the introductory section, the Hadamard
product of two stable (real or complex) polynomials does not have to be
stable. In case of real polynomials, Gregor and Bose \cite{Bos} noted that
when multiplying, in the Hadamard sense, the Hadamard product $f\circ g$ of
two Schur stable polynomials $f$ and $g$ by the polynomial $h\left(  x\right)
=\sum_{k=0}^{n}\binom{n}{k}^{-1}x^{k}$, then the product $f\circ g\circ h$,
called sometimes the Szeg\"{o} product of $f$ and $g$, becomes Schur stable.

The following theorem gives simple sufficient conditions for the Schur
stability of both the Hadamard and the Szeg\"{o} product of two complex polynomials.

\begin{theorem}
\label{thlast}Let $f,g\in\mathbb{\pi}_{n}\left(  \mathbb{C}\right)  $ be two
polynomials of the form (\ref{w01}).

\begin{enumerate}
\item[(a)] If $f$ satisfies the stability condition and $\left\vert
b_{k}\right\vert \leq1$ for $k\in N_{g}\cap N_{f}$, then both the Szeg\"{o}
product and the Hadamard product of $f$ and $g$ satisfy the stability
condition (and thus are stable).

\item[(b)] If $f$ satisfies the stability condition and $\left\vert
b_{k}\right\vert \leq\binom{n}{k}$ for $k\in N_{g}\cap N_{f}$, then the
Szeg\"{o} product of $f$ and $g$ satisfies the stability condition (and thus
is stable). In particular, if $f$ satisfies the stability condition and $g$ is
stable, then the Szeg\"{o} product of $f$ and $g$ satisfies the stability
condition (and thus is stable).

\item[(c)] If $f$ and $g$ satisfy the following condition
\[
\max\left\{  \left\vert a_{k}\right\vert ,\left\vert b_{k}\right\vert
\right\}  <\sqrt{\lambda_{k}}\text{,\qquad for every }k\in N_{f}\cap
N_{g}\text{, }%
\]
where $\left\{  \lambda_{k}\right\}  $ is a sequence of positive numbers whose
sum does not exceed~$1$, then both the Szeg\"{o} product and the Hadamard
product of $f$ and $g$ satisfy the stability condition (and thus are stable).
\end{enumerate}
\end{theorem}

Instead of the proof, which is a simple consequence of the stability
condition, we make some remarks.

Firstly, note that the assumptions on $g$ in Theorem~\ref{thlast}.(a) and
Theorem~\ref{thlast}.(b) do not imply its stability. It means that for the
Schur stability of the Hadamard product or the Szeg\"{o} product of two
complex polynomials $f$ and $g$, it suffices to require slightly more than the
stability of $f$ and slightly less than the stability of $g$. As we know, the
stability of $f$ and $g$ does not suffice.

Note also, that the assumption on $f$ and $g$ in Theorem~\ref{thlast}.(c) does
not guarantee their stability. Theorem~\ref{thlast}.(c) can be thus viewed as
a sufficient condition for the stability of the Hadamard product and the
Szeg\"{o} product of two (unstable) polynomials.

We close this part with the following conclusion (its simple proof based on
the stability condition is omitted).

\begin{conclusion}
\label{thlast2} For every non-zero polynomial $f\in\mathbb{\pi}_{n}\left(
\mathbb{C}\right)  $ there exists a stable polynomial $g\in\mathbb{\pi}%
_{n}\left(  \mathbb{C}\right)  $ such that both the Szeg\"{o} product and the
Hadamard product of $f$ and $g$ are stable.
\end{conclusion}

\subsection{Does it work for polynomials of fractional orders?}

At the end, let us note that all the above results can also be applied to
fractional-order polynomials.

Recall that a fractional-order polynomial is a function of the form%
\begin{equation}
f:s\rightarrow s^{\sigma_{n}}+a_{n-1}s^{\sigma_{n-1}}+\ldots+a_{1}%
s^{\sigma_{1}}+a_{0}\text{,} \label{w07}%
\end{equation}
where $a_{0},\ldots,a_{n-1}$ are known coefficients and $\sigma_{n}%
>\sigma_{n-1}>\ldots>\sigma_{1}>0$ are known powers being real numbers. The
polynomials of non-integer order play an important role in the stability
analysis of linear time-invariant fractional-order systems (e.g. Matignon
\cite{Mat}) and have recently attracted lots of attention in the control
theory literature.

If at least one power in (\ref{w07}) is non-integer, then the fractional-order
polynomial $f$ is a multivalued function. Supposing that $\sigma_{k}=\alpha k$
for some positive number $\alpha$ ($f$ is then said to be of a commensurate
order) and substituting $s^{\alpha}=w$ in (\ref{w07}), we obtain an
integer-order polynomial $F_{f}$ associated with $f$
\[
F_{f}\left(  w\right)  =w^{n}+a_{n-1}w^{n-1}+\ldots+a_{1}w+a_{0}\text{.}%
\]
As $\alpha$ is a rational number, every root of $F_{f}$ gives a finite set of
roots of $f$ (as in (\ref{w03})). Moreover, according to $s^{\alpha}=w$, $f$
is Schur stable if and only if $F_{f}$ is. This shows that Theorems
\ref{th1}--\ref{thlast} and Conclusion \ref{thlast2} can be applied to both
integer-order and fractional-order polynomials.

\section{Numerical experiments}

In closing, we shall give two numerical examples completing and illustrating
the results developed in this work.

\begin{example}
\label{ex1}Consider two real polynomials $f\ $and $g$%
\begin{align*}
f\left(  s\right)   &  =s^{5}+0.9s^{2}+0.2s+0.7,\\
g\left(  s\right)   &  =s^{5}+2.5s^{2}+2s+3,
\end{align*}
both having zeros outside the unit disc and thus unstable. In order to
illustrate Theorem \ref{th1} we need to approximate value (\ref{wth1}) for $f$
and value (\ref{wth2}) for $g$. The approximations were obtained by generating
sequences $\{\lambda_{k}\}$ of the form $\{ml/n^{2},m\left(  n-l\right)
/n^{2},\left(  n-m\right)  /n\}$ for $n=10^{3}$ and $m,l=1,\ldots,n-1$, and
performing necessary computations. The approximation of (\ref{wth1}) for $f$
is $p_{\max}^{\ast}\approx3.40372$, whereas the minimal value of $p^{\ast}$
such that $f^{\left[  p\right]  }$ is stable for every $p>p^{\ast}$ (see
Remark~\ref{rem1}) is equal to $p^{\ast}\approx3.35457$. The approximation of
(\ref{wth2}) for $g$ is $p_{\min}^{\ast}\approx-1.24121$, whereas the minimal
value of $p_{\ast}$ such that $g^{\left[  p\right]  }$ is stable for every
$p<p_{\ast}$ is equal to $p_{\ast}\approx-1.01579$.
\end{example}

\begin{figure}[tbh]
\begin{center}
\includegraphics[scale=0.6]{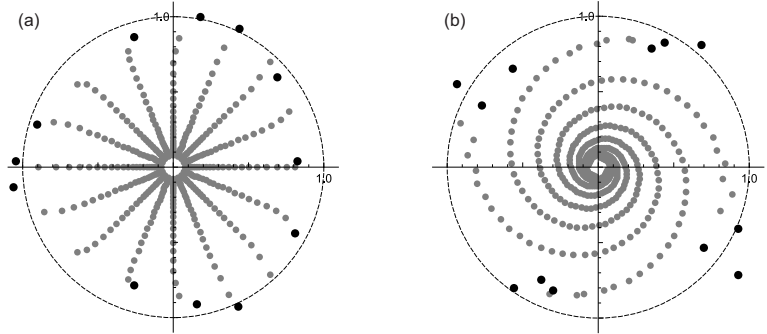}
\end{center}
\caption{(a) the zeros of $f^{\left[  p\right]  }$ marked with black dots for
$1\leq p\leq3$ and with gray dots for $4\leq p\leq100$; (b) the zeros of
$g^{\left[  q\right]  }$ marked with black dots for $-3\leq q\leq-1$ and with
gray dots for $-100\leq q\leq-4$.}%
\end{figure}

\begin{example}
Consider two complex polynomials $f\ $and $g$%
\begin{align*}
f\left(  s\right)   &  =s^{4}+\left(  0.2-0.4i\right)  s^{3}+0.7s-0.9i,\\
g\left(  s\right)   &  =s^{4}-1.5s^{3}+\left(  2-i\right)  s^{2}+1-0.5i.
\end{align*}
Proceeding as in Example~\ref{ex1} we get $p_{\max}^{\ast}\approx3.69323$ for
$f$ and $p_{\min}^{\ast}\approx-3.40696\ $for $g$. To confirm the results we
have plotted in Fig.~1 the zeros of $f^{\left[  p\right]  }$ and $g^{\left[
q\right]  }$ for integer values of $p\in\{1,\ldots,10^{2}\}$ and $q\in\left\{
-10^{2},\ldots,-1\right\}  $. The $p$-th Hadamard power of $f$ occurs unstable
for $1\leq p\leq3$ and becomes stable for $p\geq4$, as expected. Similarly,
the $q$-th Hadamard power of $g$ occurs unstable for $-1\leq q\leq-3$ and
becomes stable for $q\leq-4$.
\end{example}

\section*{Acknowledgments}

This research work was partially supported by the Faculty of Applied
Mathematics AGH UST statutory tasks (grant no. 11.11.420.004) within subsidy
of Ministry of Science and Higher Education.

\section*{Appendix A}

To prove that (\ref{eq2}) is a sufficient condition for the Schur stability of
the complex polynomial $f$ of the form
\[
f\left(  s\right)  =s^{n}+a_{n-1}s^{n-1}+\ldots+a_{1}s+a_{0}%
\]
note that
\[
\left\vert f\left(  s\right)  \right\vert \geq\left\vert s\right\vert ^{n}-%
%TCIMACRO{\tsum \nolimits_{k\in N_{f}}}%
%BeginExpansion
{\textstyle\sum\nolimits_{k\in N_{f}}}
%EndExpansion
\left\vert a_{k}\right\vert \left\vert s\right\vert ^{k}\text{,}%
\]
where $N_{f}=$ $\{k\in\left\{  0,\ldots,n-1\right\}  :a_{k}\neq0\}$. It means
that if for every $k\in N_{f}$
\[
\lambda_{k}\left\vert s\right\vert ^{n}>\left\vert a_{k}\right\vert \left\vert
s\right\vert ^{k},
\]
where $\left\{  \lambda_{k}\right\}  $ is a sequence of positive numbers whose
sum does not exceed $1$, then $f\left(  s\right)  \neq0$. In other words, if
$s$ is a zero of $f$ then
\[
\left\vert s\right\vert \leq\max_{k\in N_{f}}\left(  \frac{\left\vert
a_{k}\right\vert }{\lambda_{k}}\right)  ^{1/n-k}.
\]
If $\left\vert a_{k}\right\vert <\lambda_{k}$ for $k\in N_{f}$, then every
zero of $f$ has moduli less than $1$ and thus $f$ is Schur stable. This
completes the proof.

\end{document}